\documentclass{amsart}
\usepackage[all]{xy}
\numberwithin{equation}{section}

\usepackage{amssymb}
\let\cal\mathcal
\def\Ascr{{\cal A}}

\def\Cscr{{\cal C}}

\def\Fscr{{\cal F}}

\def\Lscr{{\cal L}}

\def\Oscr{{\cal O}}

\def\Rscr{{\cal R}}
\def\Sscr{{\cal S}}

\def\Vscr{{\cal V}}

\let\blb\mathbb

\def \ZZ{{\blb Z}}

\def\Id{\operatorname{id}}

\def\Res{\operatorname{Res}}

\def\Mod{\operatorname{Mod}}

\def\rad{\operatorname {rad}}

\def\Rep{\operatorname {Rep}}

\def\Ext{\operatorname {Ext}}
\def\Hom{\operatorname {Hom}}

\def\End{\operatorname {End}}

\def\im{\operatorname {im}}

\def\ker{\operatorname {ker}}

\def\End{\operatorname {End}}

\def\r{\rightarrow}

\def\exist{\exists}

\DeclareMathOperator{\Ind}{Ind}

\newtheorem{lemma}{Lemma}[section]
\newtheorem{proposition}[lemma]{Proposition}
\newtheorem{theorem}[lemma]{Theorem}
\newtheorem{corollary}[lemma]{Corollary}

\newtheorem{lemmas}{Lemma}[subsection]
\newtheorem{propositions}[lemmas]{Proposition}
\newtheorem{theorems}[lemmas]{Theorem}
\newtheorem{corollarys}[lemmas]{Corollary}

\theoremstyle{definition}

\newtheorem{examples}[lemmas]{Example}
\newtheorem{definitions}[lemmas]{Definition}

{

}

\theoremstyle{remark}

\newtheorem{remarks}[lemmas]{Remark}

\newdimen\uboxsep \uboxsep=1ex
\def\uboxn#1{\vtop to 0pt{\hrule height 0pt depth 0pt\vskip\uboxsep
\hbox to 0pt{\hss #1\hss}\vss}}

\def\uboxs#1{\vbox to 0pt{\vss\hbox to 0pt{\hss #1\hss}
\vskip\uboxsep\hrule height 0pt depth 0pt}}

\def\fnd{\underline{\operatorname{end}}}
\let\en\fnd
\let\pt\otimes

\def\cu{\operatorname{Cube}}
\def\cus{\operatorname{cube}}
\def\CoMod{\operatorname{CoMod}}

\let\Cube\cu

\def\Vect{\operatorname{Vect}}
\def\NRed{\operatorname{NRed}}
\title{Representations of non-commutative quantum groups}
\author{Benoit Kriegk}
\email{benoit.kriegk@univ-st-etienne.fr}
\address{Faculté des Sciences et Techniques\\
Laboratoire de Math\'ematiques de l'Universit\'e de Saint-Etienne\\
23, rue du Docteur Paul Michelon\\
42023 SAINT-ETIENNE CEDEX 02
}
\author{Michel Van den Bergh}
\email{michel.vandenbergh@uhasselt.be}
\address{Universiteit Hasselt\\ Universitaire Campus\\ 3590 Diepenbeek}
\thanks{The second author is a senior researcher at the FWO}
\keywords{Quantum groups, $N$-Koszulity, distributivity}
\subjclass{Primary 16T10; Secondary 16S37}
\begin{document}
\begin{abstract} We discuss the representation theory of the bialgebra $\fnd(A)$
introduced by Manin. As a side result we give a new proof that Koszul algebras
are distributive and furthermore we show that some well-known $N$-Koszul algebras
are also distributive.
\end{abstract}
\maketitle
\section{Introduction}
Throughout we work over a ground field $k$. If $A$ is a $\ZZ$-graded algebra then there is a bialgebra $\fnd(A)$
coacting on $A$ which is universal in an appropriate sense (see
\cite{Ma} and also \S\ref{ref-4-24} below). For example if $A=k[x,y]$ then
\begin{equation}
\label{ref-1.1-0}
\fnd(A)=k[a,b,c,d]/(ac-ca,ad-da+bc-cb,bd-db)
\end{equation}
 and the comultiplication and coaction can be
written concisely as
\begin{align*}
\Delta \begin{pmatrix}a&b\\ c&d \end{pmatrix}&=\begin{pmatrix}a&b\\ c&d \end{pmatrix}
\otimes \begin{pmatrix}a&b\\ c&d \end{pmatrix}\\
\delta \begin{pmatrix}x\\ y\end{pmatrix}&=
\begin{pmatrix}a&b\\ c&d \end{pmatrix}
\otimes \begin{pmatrix}x\\ y\end{pmatrix}
\end{align*}
We see in particular that there is a bialgebra morphism from $\fnd(A)$
to $\Oscr(M_{2\times 2})$, the coordinate ring of the monoid of
$2\times 2$-matrices. So is tempting in this case to think of the bialgebra
$\fnd(A)$ as a non-commutative version of this monoid.

Recently the bialgebra $\fnd(A)$ appeared as a vehicle for deriving
combinatorial identities such as the quantum MacMahon Master theorem
\cite{HKL,HL1}. In order to understand how far such methods can be
pushed it is useful to describe the representation theory of
$\fnd(A)$.\footnote{Unless otherwise specified, representations of a
  bialgebra are coalgebra representations.}.
In this paper we will accomplish
this for a large class of algebras.

Our results will be for so-called \emph{distributive algebras}. Let $A=TV/(R)$
with $R\subset V^{\otimes N}$. We say that $A$ is distributive if for
all $n$ the subspaces $V^{\otimes a}\otimes R\otimes V^{\otimes n-N-a}$ of
$V^{\otimes n}$ generated a distributive lattice. If $N=2$ then it is well
known that distributivity is the same as Koszulity \cite{BF} (in Theorem \ref{ref-3.1-22} below we
give an alternative proof of this result).
For $N>2$ distributivity is related to ``$N$-Koszulity'', a higher Koszul
property introduced by Berger \cite{Berger}. 

Recall that $A$ is $N$-Koszul if
the free modules occuring in the minimal resolution of $k$ are generated in degrees
$0,1,N-1,N,2N-1,2N$, etc\dots. One has the following implications
\begin{equation}
\label{ref-1.2-1}
\text{Distributivity}+\text{``extra condition"}\Rightarrow \text{$N$-Koszul}
\Rightarrow\text{``extra condition"}
\end{equation}
Here the ``extra condition'' is a suplementary condition which is vacuous
for $N=2$ and for $N=3$ reduces to
\[
R\otimes V\otimes V\cap V\otimes V\otimes R\subset V\otimes R\otimes V
\]
It is unfortunately unknown if the first implication in \eqref{ref-1.2-1} is
reversible.

To state our results we recall the definition of the category
$\Cube_n$ introduced by Polishchuk and Positselski in \cite{PP}.
$\Cube_0$ is the category of vector spaces, $\Cube_1$ is the
category of vector space homomorphisms, $\Cube_2$ is the category of
commutative squares etc\dots. In general $\Cube_n$ is given by the
representations of a hypercube $Q_n$ with commutative faces. We put
$\Cube_\bullet=\bigoplus_n \Cube_n$. Then as observed in \cite{PP}
$\Cube_\bullet$ carries a natural monoidal structure. See
\S\ref{ref-3.1-2}.

The following
is our main result concerning the representation theory of $\fnd(A)$.
\begin{proposition}  (Proposition \ref{ref-5.3-37} below)
Assume that $A$ is distributive. Then there is a monoidal functor
$\Fscr_A:\Cube_{\bullet}\r\CoMod(\fnd(A))$ which  induces
an equivalence of monoidal categories
\[
\bar{\Fscr}_A:\Cube_{\bullet}/\Sscr_A\cong \CoMod(\fnd(A))
\]
where $\Sscr_A$ is
the localizing subcategory  of $\Cube_{\bullet}$ generated by  the
simples $S$ such that $\bar{\Fscr}_A(S)=0$.
\end{proposition}
This proposition yields a description of the representation theory
of $\fnd(A)$ in terms of the subquivers of $(Q_n)_n$ whose vertices
correspond to the simples that are not annihilated by
$\Fscr_A$.

It remains to determine the simples that are killed by $\Fscr_A$. If we
assume the extra condition then this is a tractable problem.
\begin{proposition} (Corollary \ref{ref-5.4-40} below) Assume that $A$
  satisfies the extra condition (besides being distributive). Then the
  monoidal category $\CoMod(\fnd(A))$ is entirely determined by the
dimensions of the vector spaces
\[
 \bigcap_a V^{\otimes a}\otimes R\otimes V^{\otimes
    n-a-N}
\]  If $N=2$ then $\CoMod(\fnd(A))$ is entirely determined
  by the Hilbert series of $A$.
\end{proposition}
In \S\ref{ref-6-41} we discuss the case of the symmetric algebra. We
prove
\begin{proposition} (Proposition \ref{ref-6.1-42} below)
Let $A=SV$ where $V$ is an $n$-dimensional vector space. Then the
simple $\fnd(A)$-representations when viewed as $M_{n\times n}$ representations are
of the form $S^{\lambda/\mu}V$ with $\lambda/\mu$ running through the set of rim hooks.
\end{proposition}
Recall that a rim hook is
a connected skew diagram containing no $2\times 2$-squares. An example is the following
\setlength{\unitlength}{0.5cm}
\begin{equation}
\begin{picture}(6,6)(0,0)
\put(0,0){\line(1,0){3}}
\put(0,0){\line(0,1){1}}
\put(3,0){\line(0,1){3}}
\put(3,3){\line(1,0){3}}
\put(6,3){\line(0,1){3}}
\put(0,1){\line(1,0){2}}
\put(2,1){\line(0,1){3}}
\put(2,4){\line(1,0){3}}
\put(5,4){\line(0,1){2}}
\put(5,6){\line(1,0){1}}
\end{picture}
\end{equation}
In the last section of this paper we discuss some examples of
$N$-Koszul algebras which we can show to be distributive. It
is noteworthy that we do not know an example of an $N$-Koszul algebra
which is not distributive. On the other hand we also do not know (for
$N>2$) a homological characterization of distributivity.

We show distributivity of the following algebras
\begin{itemize}
\item Algebras derived from Koszul algebras. This is more or less a
  tautology but it gives a quick proof that the generalized symmetric
  algebra $TV/\wedge^N V$ introduced by Berger is $N$-Koszul.
\item The deformed Yang-Mills algebras introduced by Connes and Dubois-Violette in
\cite{CDV1,CDV2}. We show that these algebras are in fact ``confluent'' (see \cite{Berger2} and
also \S\ref{ref-7.2-47} below).
\item The 3-dimensional cubic Artin-Schelter regular algebras of ``Type A''.
This proof also relies on confluence but in a more sophisticated way. We rely
on a generalization of the $I$-type condition which was introduced in~\cite{TVdB}.
\end{itemize}

\section{The category $\Cube_n$}
\subsection{Generalities}
\label{ref-3.1-2}
The quiver $Q_{n}$ is the quiver with relations
which has a vertex $x_I$ for every  $I\subset
\{1,\ldots,n\}$ and arrows $x_{IJ}:x_I\r x_J$ for every $I\subset
J$.  The relations are given by $x_{JK}x_{IJ}=x_{IK}$ (we write paths
in functional notation). If $n\le 0$ then we declare
$Q_{n}$ to be a single vertex $x_\emptyset$ without arrows.

We denote the projective $Q_n$ representation corresponding to $x_I$
by $P^{(n)}_I$ and the corresponding simple representation by
$S_I^{(n)}$. The category of $Q_{n}$-representations (not
necessarily finite dimensional) is denoted by $\cu_n$. The full
subcategory consisting of finite dimensional representations is
denoted by $\cus_n$. An object in $\Cube_n$ is generally written as
$(X_I)_I$ where $I$ runs through the subsets of $\{1,\ldots,n\}$. We
recall the following result.
\begin{theorems} \label{ref-3.1.1-3} \cite[Lemma 9.1]{PP} A collection of subobjects
  $R_1,\ldots,R_n\subset X$ of an object $X$ in a $k$-linear abelian
  category $\Cscr$ generates a distributive lattice iff
the right exact functor $F:\cus_n\r \Cscr$ defined by $F(P_\emptyset)=X$ and
  $F(P_{\{i\}}\r P_\emptyset)=R_i\hookrightarrow X$ is exact.
\end{theorems}
Note that if $\Cscr$ is
a Grothendieck category then the functor $F$ extends to a right exact functor
$F:\cu_n\r \Cscr$ commuting with filtered exact limits. If $F:\cus_n\r \Cscr$ is exact
then so is its extension.

As observed in \cite{PP} the categories $\Cube_n$
carry an interesting monoidal structure.
We first consider the bifunctor ($m,n\ge 0$)
\[
\boxtimes:\Cube_m\times \Cube_n\r \Cube_{m+n}:((X_I)_I,(Y_J)_J)
\mapsto (X_I\otimes Y_J)_{I\cup J+m}
\]
where $J+m\subset \{m+1,\ldots,m+n\}$ is the translation of $J$ by
$m$. This tensor product is clearly associative and biexact.  In
addition we have $ P^{(m)}_I\boxtimes P^{(n)}_J=P^{(m+n)}_{I\cup
J+m} $.  For a number $N\ge 2$ (which will remain constant) one also
defines a shifted tensor product
\[
\otimes:\Cube_m\times \Cube_n\r \Cube_{m+n+N-1}:((X_I)_I,(Y_J)_J)
\mapsto (X_I)_I\boxtimes P^{(N-1)}_{\emptyset}\boxtimes (Y_J)_J
\]
This tensor product is still associative. On the level of projectives
we have
\begin{equation}
\label{ref-3.1-4}
P_I^{(m)}\otimes P_J^{(n)}=P^{(m+n+N-1)}_{I\cup (J+(m+N-1))}
\end{equation}
We extend this formula beyond its original context by declaring it to
be valid for $m\ge -N+1$, $n\ge -N+1$ where we recall that by our conventions if
$m\le 0$ then $P^{(m)}_I$ exists only for $I=\emptyset$.  Concretely
we have for $m=-N+1,\ldots,0$, $n\ge 0$
\[
\otimes:\Cube_m \times \Cube_n\r \Cube_{m+n+N-1}:(P^{(m)}_\emptyset,(Y_J)_J)
\mapsto P^{(m+N-1)}_\emptyset\boxtimes (Y_J)_J
\]
for $m\ge 0$, $n=-N+1,\ldots,0$:
\[
\otimes:\Cube_m\times \Cube_n\r \Cube_{m+n+N-1}:((X_I)_I,P^{(n)}_\emptyset)
\mapsto (X_I)_I\boxtimes P^{(n+N-1)}_\emptyset
\]
and for $m=-N+1,\ldots,0$, $n=-N+1,\ldots,0$:
\[
\otimes:\Cube_m \times \Cube_n\r \Cube_{m+n+N-1}:(P^{(m)}_\emptyset,P^{(n)}_\emptyset)
\mapsto P^{(m+n+N-1)}_\emptyset
\]
If $(\Ascr_i)_i$ is a family of abelian categories then we denote by $\bigoplus_i \Ascr_i$
the abelian category whose objects are formal direct sums $\bigoplus_i A_i$ with
$A_i\in \Ascr_i$. The $\Hom$-sets are given by
\[
\Hom_\Ascr\left(\bigoplus_i A_i,\bigoplus_i B_i\right)=\prod_i \Hom_{\Ascr_i}(A_i,B_i)
\]
By $\bigoplus_i'\Ascr_i$  we mean the full subcategory of $\bigoplus_i \Ascr_i$ where
we only consider objects $(A_i)_i$ with a finite number of non-zero $A_i$.

Put $\Cube_{\bullet}=\bigoplus_{n\ge -N+1} \Cube_n$ and
$\cus_{\bullet}=\bigoplus'_{n\ge -N+1}\cus_n$. With the tensor product defined
above $\Cube_\bullet$ and $\cus_\bullet$ are monoidal categories with a biexact
tensor product.  The unit object is $P^{(-N+1)}_\emptyset$.
\subsection{Tensor products of simples}
For the benefit of the reader we discuss the tensor product of simples
in the  category $\Cube_\bullet$.

\begin{propositions}
\label{ref-3.2.1-5}
Let $I\subset \{1,\ldots,m\}$, $J\subset
  \{1,\ldots,n\}$ with the convention that if $m\le 0$ or $n\le 0$ then
correspondingly $I=\emptyset$ or $J=\emptyset$. Put $J'=J+m+N-1$. The simples occuring in
  the Jordan-Holder series of $S_{I}^{(m)}\otimes S_{J}^{(n)}$ are
  $S^{(m+n+N-1)}_K$ with $K=I\cup J'\cup C$ for $C\subset \{\max(1,m+1),\ldots
  ,\min(m+N-1,m+n+N-1)\}$. The simples occur with multiplicity
one and are ordered according to $|K|$ with
  $S_{I\cup J'}^{(m+n+N-1)}$ appearing on top.
\end{propositions}
\begin{proof}
We will consider the case $m,n\ge 1$. The other cases are similar.
We recall the formula for the tensor product. Let $X\in \Cube_m$, $Y\in
\Cube_{n}$. Then
\[
X\otimes Y=X\boxtimes P^{(N-1)}_\emptyset\boxtimes Y
\]
If we apply this with $X=S_{I}^{(m)}$, $Y=S_{J}^{(n)}$ we
find for $K\subset \{1,\ldots, m\}$
\[
\bigl(S_{I}^{(m)}\otimes S_{J}^{(n)}\bigr)_K=
\begin{cases}
k&\text{if $K=I\cup J'\cup C$ with $C\subset \{m+1,\ldots
  m+N-1\}$}\\
0&\text{otherwise}
\end{cases}
\]
So the simples that occur in the Jordan-Holder series of
$S_{I}^{(m)}\otimes S_{J}^{(n)}$ are indeed as indicated.
\end{proof}
\subsection{Koszulity and \mathversion{bold} $\Cube_\bullet$}
It was observed in \cite{PP} that when $N=2$ the category $\Cube_\bullet$ is
intimately connected with the Koszul propery for graded algebras.
Similarly, as we show in this section, when $N>2$, the category $\Cube_\bullet$
is connected with the $N$-Koszul property introduced by Roland Berger in \cite{Berger}.

Let $N\geq 2$ and consider
$A=T(V)/(R)$ with $V$ a finite dimensional $k$-vectorspace and
$R\subseteq V^{\pt N}$. We put
\[
R_i^{(n)}=V^{\otimes i-1}\pt R \pt V^{\otimes n-N-i+1}\subset V^{\otimes n}
\]
and for $I\subset \{1,\ldots,n-N+1\}$
\[
R_I^{(n)}=\bigcap_{i\in I} R_i^{(n)}
\]
We define $F_{A}$ as the right exact
functor commuting with filtered colimits
\[
F_{A}:\Cube_{n-N+1}\r \Vect
\]
which sends $P^{(n-N+1)}_I$ to $R^{(n)}_I$.
As usual we denote its left derived
functor by $LF_{A}$. It is easy to see that $F_{A}$ is a monoidal functor.

Let $\Lscr_n(A)$ be the lattice of subspaces of $V^{\otimes n}$ generated
by $(R_i^{(n)})_i$.
\begin{definitions}
\label{ref-3.3.1-6}
$A$ is \emph{distributive} if for each $n$ the
lattice $\Lscr_n(A)$ is distributive.
\end{definitions}
\begin{definitions} \label{ref-3.3.2-7}
\cite{Berger} $A$ satisfies the \emph{extra condition} if
for all $1\le i\le l\le j\le n$ with $j\le i+N-1$ we have
\[
R^{(n)}_i\cap R^{(n)}_j\subset R^{(n)}_l
\]
\end{definitions}
For use below we also put for $n\ge N$
\[
J_n=R_1^{(n)}\cap\cdots \cap R_{n-N+1}^{(n)}
\]
Thus $J_N=R$. For $n\le N-1$ we put $J_n=V^{\otimes n}$. We denote the composition
\begin{equation}
\label{ref-3.2-8}
A\otimes J_{n}\r A\otimes V\otimes J_{n-1}\r A\otimes J_{n-1}
\end{equation}
by $\delta$. Here the first map is derived from the inclusion $J_{n}\subset
V\otimes J_{n-1}$ and the second map  is derived from the multiplication $A\otimes V\r A$.
Note that $\delta^N=0$. Thus $(A\otimes J_n,\delta)$ is an $N$-complex \cite{Berger1}.
\begin{definitions}\label{ref-3.3.3-9}
 \cite{Berger} $A$ is $N$-Koszul if the following complex
\begin{equation}
\label{ref-3.3-10}
\cdots \r A\otimes J_{2N} \xrightarrow{\delta^{N-1}} A\otimes J_{N+1}\xrightarrow{\delta} A\otimes
R\xrightarrow{\delta^{N-1}} A\otimes V\xrightarrow{\delta} A\r k
\end{equation}
is exact.
\end{definitions}
A famous theorem by Backelin \cite{BF} states that if $N=2$ then $A$ is distributive
if and only if it is Koszul (see below for a proof in the current setting). In general
Berger proved the implications \cite{Berger}
\begin{equation}
\label{ref-3.4-11}
\text{Distributivity}+\text{extra condition}\Rightarrow \text{$N$-Koszul}
\Rightarrow\text{extra condition}
\end{equation}
(the extra condition is vacuous if $N=2$). It is unknown if the converse
of the first implication holds.

The category $\Cube_\bullet$ contains a preimage under $F_A$ of the
degree $m$ part of the Koszul complex of $A$ \eqref{ref-3.3-10}.
Put
\[
W_{m,n}=S^{(m-N+1)}_{\emptyset}\otimes P^{(n-N+1)}
_{\{1,\ldots,n-N+1\}}
\]
Since $A_m=F_{A}(S^{(m-N+1)}_\emptyset)$ and $J_n=F_A(P^{(n-N+1)}
  _{\{1,\ldots,n-N+1\}})$ (with the usual caveat if $n-N+1\le 0$)
one obtains
\[
F_A(W_{m,n})=A_m\otimes J_n
\]
A quick computation shows
\begin{equation}
\label{ref-3.5-12}
(W_{m,n})_I=\begin{cases}
k&\text{if $\{m+1,\ldots, m+n-N+1\}\subset I\subset \{m-N+2,\ldots,m+n-N+1\}$}\\
0&\text{otherwise}
\end{cases}
\end{equation}
and all maps are either the identity (when they are between copies of $k$) or zero.
Define
\[
\delta:W_{m,n}\r W_{m+1,n-1}
\]
such that $\delta_I$ is the identity on $k$ when $(W_{m,n})_I=(W_{m-1,n+1})_I=k$ and zero
otherwise. One checks that this is indeed a well defined map in $\Cube_\bullet$ and
that $F_A(\delta)=\delta_m$ where $\delta_m$ is the degree $m$ part of
the map \eqref{ref-3.2-8}. Furthermore $\delta^N=0$.  It follows that the following complex
in $\Cube_\bullet$
\begin{equation}
\label{ref-3.6-13}
\cdots\xrightarrow{\delta} W_{m-2N,2N}\xrightarrow{\delta^{N-1}}W_{m-N-1,N+1} \xrightarrow{\delta}W_{m-N,N}\xrightarrow{\delta^{N-1}} W_{m-1,1}\xrightarrow{\delta} W_{m,0}\r
S_\emptyset^{(-N+1)}\cdot \delta_{m,0}
\end{equation}
maps to the degree $m$ part of the Koszul complex.

Let us say that $I\subset \{1,\ldots,m\}$
has a hole of size $u$ (with $u\ge 1$) if there exist $i<j\in I$ such
that $j-i=u+1$ and for all $i<l<j$ we have $l\not\in I$.
\begin{propositions}
\label{ref-3.3.4-14}
The homology of \eqref{ref-3.6-13} is an extension of simples $S^{(m-N+1)}_K$
such that $K$ contains holes of size $\le N-2$.
\end{propositions}
\begin{proof}
to prove this we will more generally consider
the homology of the complexes
\begin{equation}
\label{ref-3.7-15}
W_{m,n}\xrightarrow{\delta^a} W_{m+a,n-a}\xrightarrow{\delta^{N-a}} W_{m+N,n-N}
\end{equation}
for any $0<a<N$. \eqref{ref-3.6-13} is built up from such complexes for $a=1,N-1$.

Let $H$ be the homology of \eqref{ref-3.7-15}. From
\eqref{ref-3.5-12} we deduce that $H_I$ is non-zero iff the following conditions
hold
\begin{equation}
\label{ref-3.8-16}
\begin{gathered}
I\not\subset \{m+2,\ldots,m+n-N+1\}\\
\{m+1,\ldots,m+n-N+1\}\not\subset I\\
\{m+a+1,\ldots, m+n-N+1\}\subset I\subset \{m+a-N+2,\ldots,m+n-N+1\}
\end{gathered}
\end{equation}
Let us assume that all sets occurring in these conditions are non-empty.
This happens when
\begin{equation}
\label{ref-3.9-17}
N\le n-a
\end{equation}
If $I$ satisfies \eqref{ref-3.8-16}\eqref{ref-3.9-17} then we have the following:
\begin{gather*}
\exist i:m+a-N+2\le  i\le m+1:i\in I\\
\exist j:m+a \ge  j\ge m+1:j\not\in I
\end{gather*}
Consider the triple $\{i,j,m+a+1\}$. Here $i,m+a+1\in I$ and $j\not\in I$.
We clearly have $i<j<m+a+1$ and one computes
$m+a+1-i\le N-1$. Hence $I$ contains a hole of size $\le N-2$.

Looking at the condition \eqref{ref-3.9-17} we see that the homology of \eqref{ref-3.6-13}
is an extension of simples $S^{(l)}_K$
such that $K$ contains holes of size $\le N-2$, except perhaps at $W_{m-1,1}$ and $W_{m,0}$.

To compute the homology at $W_{m-1,1}$ we have to make the substitutions
$m\rightarrow m-N$, $n\rightarrow N$, $a\rightarrow N-1$ in \eqref{ref-3.7-15}.

Thus the condition \eqref{ref-3.8-16} becomes
\begin{equation}
\label{ref-3.10-18}
\begin{gathered}
I\not\subset \{m-N+2,\ldots,m-N+1\}=\emptyset\\
\{m-N+1,\ldots,m-N+1\}=\{m-N+1\}\not\subset I\\
\{m,\ldots, m-N+1\}=\emptyset\subset I\subset \{m-N+1,\ldots,m-N+1\}=\{m-N+1\}
\end{gathered}
\end{equation}
The only $I$'s satisfying the last condition are $\emptyset$,
$\{m-N+1\}$. $I=\emptyset$ does not satisfy the first condition and
$I=\{m-N+1\}$ does not satisfy the second condition. Hence no $I$
satisfies all three. Thus the homology at $W_{m-1,1}$ is zero. 

We leave to the reader the easy verification that the homology of
\eqref{ref-3.6-13} at $W_{m,0}$ is also zero.
\end{proof}
\begin{remarks} A perhaps better way to understand the complex \eqref{ref-3.6-13}
is as follows. Write $\Vscr=P^{(-N+2)}_\emptyset\in \Cube_{-N+2}$ and
$\Rscr=P^{(1)}_{\{1\}}\in \Cube_1$. Then there is a unique non-zero map (up to
scalar) in $\Cube_1$
\[
\Rscr\r \Vscr^{\otimes N}
\]
Let $\Ascr=T\Vscr/(\Rscr)$ be the corresponding graded algebra object in
the monoidal abelian category $\Cube_\bullet$. Then $F_A(\Ascr)=A$.
The Koszul complex \eqref{ref-3.3-10}
can be defined in an arbitrary abelian monoidal category and in particular it can
be defined for $\Ascr$. Then \eqref{ref-3.6-13} is just the
degree $m$ part of this generalized Koszul complex.

As we have seen the Koszul complex for $\Ascr$ is not exact. Define
\[
\Cube^{\text{ec}}_\bullet=\Cube_\bullet/\Sscr^{\text{ec}}
\]
where $\Sscr^{\text{ec}}$ is the localizing subcategory generated by the
simples $S_{K}^{(l)}$ such that $K$ has a hole of size $\le N-2$. It
is easy to see that $\Cube^{\text{ec}}_\bullet$ inherits the monoidal structure of
$\Cube^{\text{ec}}_\bullet$.

Then the image of the Koszul complex of $\Ascr$ in $\Cube^{\text{ec}}_\bullet$ is exact.
Thus $\Ascr$ is an $N$-Koszul algebra in the monoidal abelian category $\Cube^{\text{ec}}_\bullet$.
\end{remarks}

The properties introduced in Definitions
\ref{ref-3.3.1-6}-\ref{ref-3.3.3-9} translate nicely into properties
of the functor $F_{A}$. We say that an object $M\in \Cube_m$ is
acyclic if $L_iF_AM=0$ for $i>0$. Note the following
\begin{lemmas}
\label{ref-3.3.6-19}
If $M\in \Cube_a$ and $N\in \Cube_b$ are acyclic then so is $M\otimes N$.
\end{lemmas}
\begin{proof} Let $P_\bullet\r M\r 0$ and $Q_\bullet \r N\r 0$ be
  minimal projective resolutions of $M,N$. The exactness of the tensor
  product implies that $P_\bullet\otimes Q_\bullet$ is a minimal
  projective resolution of $M\otimes N$. Hence $L_\bullet F_A(M\otimes
  N)$ is computed by $F_A(P_\bullet\otimes Q_\bullet)=
  F_A(P_\bullet)\otimes F_A(Q_\bullet)$. Since $M$ and $N$ are acyclic,
$F_A(P_\bullet)$, $F_A(Q_\bullet)$ are resolutions of $F_AM$, resp.\ $F_AN$. Hence
$F_A(P_\bullet)\otimes F_A(Q_\bullet)$ is a resolution of $F_AM\otimes F_AN$.
\end{proof}

The following is the main result of this section.

\begin{propositions}
\label{ref-3.3.7-20}
\begin{enumerate}
\item
$A$ is distributive if and only if the functor $F_{A}$ is exact for all $n$,
i.e.\ if and only if all $S^{(n-N+1)}_I$ are acyclic.
\item $A$ satisfies the extra condition if and only if for all $n$ one
  has $F_{A}(S^{(n-N+1)}_I)=0$ for those $I$ which have a hole of size $\le N-2$.
\item Assume that $A$ satisfies the extra condition. Then for all $n$ and all $I$ that
contain a hole of size $\le N-2$ one has that $S^{(n-N+1)}_I$ is acyclic.
\item Assume that $A$ satisfies the extra condition. Then
$A$ is $N$-Koszul if and only if for all $n$ one has that $S_{\emptyset}^{(n-N+1)}$ is
acyclic.
\end{enumerate}
\end{propositions}
\begin{proof}
\begin{enumerate}
\item This is just a translation of Theorem \ref{ref-3.1.1-3}.
\item The simple $S_I^{(n-N+1)}$ has a presentation
\[
\bigoplus_{l\not\in I} P_{I\cup \{l\}}^{(n-N+1)}\r P_I^{(n-N+1)}\r S_I^{(n-N+1)}\r 0
\]
and thus
\[
F_{A}(S_I^{(n-N+1)})=R_I^{(n)}\left/\sum_{l\not\in I} R^{(n)}_{I\cup \{l\}}\right.
\]
Assume that the extra condition holds and that $I$ has a hole
of size $\le N-2$, delimited by $i<j$. Pick $i<l<j$. Then $R^{(n)}_I=R^{(n)}_{I\cup \{l\}}$.
Hence $F_{A}(S^{(n-N+1)}_I)=0$.

Conversely assume $F_{A}(S^{(n-N+1)})=0$ for all $I$ containing a hole of size $\le N-2$.
Then for such an $I$ we have
\[
R_I^{(n)}=\sum_{l\not\in I} R^{(n)}_{I\cup \{l\}}
\]
Iterating this identity we get
\[
R_I^{(n)}=\sum_{i=1}^p R_{J_i}^{(n)}
\]
where $(J_i)_i$ is the set of all subsets of $\{1,\ldots,n\}$ which contain $I$
and which have no holes of size $\le N-2$.

Assume now $I=\{i,j\}$ with $i<j\le i+N-1$ and take $i<l<j$. Each of the $J_i$
must contain $l$. Hence $R^{(n)}_{J_i}\subset R^{(n)}_l$ and thus
$R^{(n)}_i\cap R^{(n)}_j=R^{(n)}_I\subset R^{(n)}_l$. Hence the extra condition holds.
\item Assume that $I$ contains a hole of length $\le N-2$, delimited
  by $i<j$. Pick $i<l<j$.  The minimal resolution $K_\bullet$ of
  $S^{(n+N-1)}_I$ has the form $\bigoplus_{K\supset I}
  P_K^{(n-N+1)}$. The $P_K^{(n-N+1)}$ come in pairs $P^{(n-N+1)}_{K_1}, P^{(n-N+1)}_{K_2}$ with
  $l\not \in K_1$ and $K_2=K_1\cup \{l\}$. Since the extra condition
  implies $F_A(P^{(n-N+1)}_{K_1})=F_A(P^{(n-N+1)}_{K_2})$ we easily deduce that $F_A(K_\bullet)$
  is acyclic.
\item
Assume that $A$ is $N$-Koszul (and hence satisfies the extra
condition by \eqref{ref-3.4-11}). We will show by induction that
$S_{\emptyset}^{(m-N+1)}$ is acyclic. For $m=0$ this is trivial. So
assume $m>0$. Then \eqref{ref-3.6-13} yields a (finite complex)
\begin{equation}
\label{ref-3.11-21}
\cdots \r W_{m-N,N}\r W_{m-1,1} \r S_\emptyset^{(m-N+1)} \r 0
\end{equation}
By induction and Lemma \ref{ref-3.3.6-19} $W_{m-a,a}$ is acyclic for $a>0$.
By (2,3) the homology of \eqref{ref-3.11-21} is acyclic as well and killed by $F_A$. Hence $L_i F_A S_\emptyset^{(m-N+1)}$
for $i>0$ is computed by the homology of.
\[
\cdots \r F_AW_{m-N,N}\r F_AW_{m-1,1} \r  S_\emptyset^{(m-N+1)}\r 0
\]
This latter complex is precisely the degree $m$ part of the Koszul
complex \eqref{ref-3.3-10} which is acyclic by our hypothesis that
$A$ is $N$-Koszul.

Conversely assume that $S_\emptyset^{(m-N+1)}$ is acyclic for all
$m$. Then the terms of \eqref{ref-3.6-13} are acyclic. Since the
homology of \eqref{ref-3.6-13} is acyclic and killed by~$F_A$ (by (2,3))
it follows that the complex \eqref{ref-3.6-13} becomes exact after
applying~$F_A$. The resulting exact sequence is precisely the degree $m$ part
of \eqref{ref-3.3-10}. Hence combining all $m$ we find that $A$ is $N$-Koszul. \qed
\end{enumerate}
\def\qed{}
\end{proof}
\begin{theorem}\cite{BF} \label{ref-3.1-22} If $N=2$ then $A$ is Koszul if and only if
  it is distributive.
\end{theorem}
\begin{proof} Assume $N=2$. If $A$ is distributive then by applying
$F_A$ to the complex \eqref{ref-3.6-13} we obtain the degree $m$
part of the Koszul complex of $A$.  Using (1,2,3) of Proposition \ref{ref-3.3.7-20} we
see this complex is exact.

We now prove the ``difficult''
  direction. Assume that $A$ is Koszul. By Proposition \ref{ref-3.3.7-20} we have
to prove that all $S^{(n-1)}_I$ are acyclic. By \ref{ref-3.3.7-20}(4)
  we know that all $S^{(n-1)}_{\emptyset}$ are acyclic.

We will
use induction on $|I|$ and $n$ to prove that $S^{(n-1)}_I$ is
acyclic for all $I$. If $I=\emptyset$ then there is nothing to prove.
So assume $l\in I$. Put $I_1=I\cap \{1,\ldots,l-1\}$ and
$I_2=(I\cap \{l+1,\ldots,n-1\})-l$ ($I_1$, $I_2$ can be empty). Then
by Proposition \ref{ref-3.2.1-5}  we have an exact sequence (see also \cite{PP})
\begin{equation}
\label{ref-3.12-23}
0\r S_{I}^{(n-1)} \r S_{I_1}^{(l-1)}\otimes S_{I_2}^{(n-l-1)}\r S_{I\setminus \{l\}}^{(n-1)}
\r 0
\end{equation}
By induction $S_{I_1}^{(l-1)}$, $S_{I_2}^{(n-l-1)}$, $S_{I\setminus \{l\}}^{(n-1)}$
are all acyclic. Then \eqref{ref-3.12-23} implies that $S_I^{(n-1)}$ is acyclic
as well.
\end{proof}
\begin{remarks} The first part of the previous proof works for arbitrary $N$ and yields
the first implication in \eqref{ref-3.4-11}. In fact for a distributive algebra
satisfying the extra condition, the Koszul complex \eqref{ref-3.3-10} is
only one of many long exact sequences one may construct.

To be more precise we have already noted that $(A\otimes J_n,\delta)_n$ is
an $N$-complex (see \cite{Berger1}). An $N$-complex
\[
\cdots \r Y^i\xrightarrow{\delta} Y^{i+1}\xrightarrow{\delta} Y^{i+2}\xrightarrow{\delta}
\]
can be contracted into
a number of genuine complexes
\[
\cdots \r Y^i\xrightarrow{\delta^a} Y^{i+a}\xrightarrow{\delta^{N-a}} Y^{i+N}
\xrightarrow{\delta^a} Y^{i+N+a}\r \cdots
\]
By analyzing the proof of Proposition \ref{ref-3.3.4-14} (in
particular the properties of the complexes \eqref{ref-3.7-15} when
\eqref{ref-3.9-17} holds) we find that the contracted complexes of
$(A\otimes J_n,\delta)$ are  exact, except in their initial degrees.
For example if $N=3$ then the complex
\[
\cdots \r A\otimes J_{N+2}\xrightarrow{\delta^2} A\otimes R \xrightarrow{\delta} A\otimes V\otimes V\xrightarrow{\delta^2} A
\]
becomes exact starting at the term $A\otimes R$.

On the contrary, the contracted complexes are almost never exact in their initial
degrees. This has been observed in \cite{Berger1}.
\end{remarks}

\section{The bialgebra  $\fnd(A)$}
\label{ref-4-24} Let $A$ be a $\ZZ$-graded algebra. It is easy to
see that there exists an algebra $B$ which is universal for the
 property that there exists an algebra morphism $\delta:A\r B\otimes A$ such
that $\delta(A_n)\subset B\otimes A_n$ for all $n$. In other words
for any algebra morphism $\partial:A\r C\otimes A$ which preserves
the $A$-grading there is a unique algebra morphism $\gamma:B\r C$
such that $\partial=(\gamma\otimes\Id)\circ \delta$. From the
universality property one immediately obtains a bialgebra structure
on $B$. Following Manin \cite{Ma} we use the notation $\fnd(A)$ for
$B$.

We consider a special case of this construction.  Let $N\geq 2$ and consider
$A=T(V)/(R)$ with $V$ a finite dimensional $k$-vectorspace and
$R\subseteq V^{\pt N}$. Then it is not hard to see that
\[
\en(A)=T(V^*\pt V)/(\pi_N(R^\perp \pt R)),
\]
where $\pi_N$ is the shuffle map
\begin{multline*}
\pi_N: (V^*\pt V)^{\pt N} \r V^{*\pt N}\pt V^{\pt N} :
\\ f_1\pt \cdots \pt f_N \pt x_1 \pt \cdots \pt x_N \mapsto f_1\pt x_1 \pt \cdots \pt f_N \pt x_N
\end{multline*}
The co-action of $\fnd(A)$ on $A$, as well as the bialgebra structure on
$\fnd(A)$ can be described concretely in terms of generators. Fix a $k$-basis
$(x_i)_i$ for $V$ and let $(x_i^*)_i$ be its dual basis in $V^*$. Put
$z^{j}_i=x^\ast_j\otimes x_i$. Then the co-action of $\fnd(A)$ on $A$ is
given in terms of generators by
\begin{equation}
\label{ref-4.1-25}
\delta(x_i)=\sum_k z_i^k\otimes x_k
\end{equation}
and the bialgebra structure on $\fnd(A)$ is given by
\begin{equation}
\label{ref-4.2-26}
\begin{aligned}
\Delta(z^j_i)&=\sum_k z^k_i\otimes z_k^j\\
\epsilon(z^j_i)&=\delta^j_i
\end{aligned}
\end{equation}
Below we will be interested in the finite dimensional
co-representations of the bialgebra $\fnd(A)$. We first note that
$\fnd(A)$ is a graded algebra by putting $\deg z^j_i=1$. We denote its part
of degree $n$ by $\fnd(A)_n$. It follows immediately from \eqref{ref-4.2-26} that
$\fnd(A)_n$ is stable under $\Delta$ and when equipped with the restriction of
$\epsilon$ it becomes a coalgebra.

We now recall some facts about sums of coalgebras.
Let $(C_i,\Delta_i,\epsilon_i)_i$ be a family of coalgebras. Put $C=\bigoplus_i C_i$.
Then $C$ is a coalgebra with comultiplication $\Delta(c)=\bigoplus_i \Delta_i(c_i)$
and counit $\epsilon(c)=\sum_i \epsilon_i(c_i)$ for $c=\bigoplus_i c_i$ with $c_i\in C_i$.

Suppose we are given comodules $(W_i,\delta_i)$ for $C_i$. Then
clearly $W=\bigoplus_i W_i$ is a $C$-comodule with coaction
$\delta(w)=\bigoplus_i\delta_i(w_i)$ for $w=\bigoplus_i w_i$, $w_i\in
W_i$.
\begin{lemma} \label{ref-4.1-27} The functor
\[
\Psi:\bigoplus_i \operatorname{CoMod}(C_i)\r \operatorname{CoMod}(C):\bigoplus_i W_i\mapsto
\bigoplus_i W_i
\]
is an equivalence of categories.
\end{lemma}
\begin{proof} This is well known. If $W\in \operatorname{CoMod}(C)$ then it decomposes
as $(C_i)_i$-comodules via $W=\bigoplus_i p_i(W)$ where $p_i(w)=\sum_w \epsilon_i(w_{(1)})w_{[2]}$ for $\delta(w)=\sum_w w_{(1)}\otimes w_{[2]}$.
\end{proof}
From this proposition we deduce that in order to understand the representation
theory of $\fnd(A)$ it is sufficient to understand the representation theory
of $\fnd(A)_n$.
Since $\fnd(A)_n$ is finite dimensional we have
\[
\CoMod(\fnd(A)_n)\cong\Mod(\fnd(A)^{\ast\circ}_n)
\]
Hence we have to describe the algebras $\fnd(A)^{\ast\circ}_n$. We do this next. For
$n<N$ define
\[
Z_n(A)=\End_k(V^{\pt n})
\]
and for $n\ge N$
\[
Z_n(A)=\left\{\varphi \in \End_k(V^{\pt n}) \mid \forall i\in
\{1,\ldots,n-N+1\}: \varphi(R_i^{(n)})\subset R_i^{(n)} \right\}
\]
where $R_i^{(n)}=V^{\otimes i-1}\pt R \pt V^{\otimes n-N-i+1}\subset V^{\otimes n}$.
\begin{proposition}
\label{ref-4.2-28}
We have as algebras
\label{ref-4.2-29}
\[
\fnd(A)^{\ast\circ}_n\cong Z_n(A)
\]
\end{proposition}
\begin{proof} We consider the most difficult case $n\ge N$. In that case
the result follows by the following computation.
\begin{align*}
\en(A)_n^*&=\left( (V^*\pt V)^{\pt n} / (\pi_N(R^\perp \pt R))_n \right)^*\\
&= \left( (V^*\pt V)^{\pt n} / \sum_{i=0}^{n-N} (V^{*}\pt V)^{i} \pt \pi_N(R^\perp \pt R) \pt (V^{*}\pt V)^{n-N-i}  \right)^*\\
&\cong \left( (V^{* \pt n}\pt V^{\pt n}) / \sum_{i=0}^{n-N} V^{* i}\pt R^\perp \pt V^{* n-N-i} \pt V^{i}\pt R \pt V^{n-N-i}  \right)^*\\
&\subset (V^{* \pt n}\pt V^{\pt n})^\ast
\end{align*}
Writing $Y_i=V^{* \otimes i}\pt R^\perp \pt V^{* \otimes n-N-i} \pt V^{\otimes i}\pt R \pt
V^{\otimes n-N-i}$ we get
\begin{equation}
\label{ref-4.3-30}
\begin{aligned}
\en(A)_n^*& \cong  \left(\sum_{i=0}^{n-N} Y_i\right)^\perp \\
& =  \bigcap_{i=0}^{n-N} Y_i^\perp\\
&= \left\{\varphi \in (V^{* \pt n}\pt V^{\pt n})^* \mid \forall i: \varphi(Y_i)=0 \right\}\\
&= \left\{\varphi \in \End_k(V^{\pt n}) \mid \forall i:
\varphi(V^{i}\pt R \pt V^{n-N-i})\subseteq V^{i}\pt R \pt V^{n-N-i}
\right\}
\end{aligned}
\end{equation}
via the canonical isomorphism $(E^\ast\pt E)^\ast \cong \End_k(E)$
(for finite dimensional $E$).

 Here we have to make the small
computation that if $X\subseteq E$ and $\Psi \in (E^\ast\pt E)^\ast$
then $\Psi(X^\perp \pt X)=0$ if and only if the corresponding element
$\varphi$ in $\End_k(E)$ satisfies $\varphi(X)\subseteq X$.

We claim that the vector space isomorphism \eqref{ref-4.3-30} is compatible
with the given algebra structures on both sides, up to exchanging factors. To verify this we may assume
that $R=0$, i.e.\ $A=TV$. Then we have to show that the following isomorphism
\begin{align*}
T(V^\ast\otimes V)^\ast_n&\cong ((V^\ast\otimes V)^{\otimes n})^\ast\\
&\cong ((V^\ast)^{\otimes n}\otimes V^{\otimes n})^\ast\\
&\cong V^{\otimes n}\otimes (V^\ast)^{\otimes n}\\
&\cong \End(V^{\otimes n})^\circ
\end{align*}
is compatible with the algebra structure. We verify this using the explicit bases introduced
above. The comultiplication on $T(V^\ast\otimes V)_n$ is given by
\[
\Delta(z^{j_1}_{i_i}\cdots z^{j_n}_{i_n})=\sum_{k_1\cdots k_n} z^{k_1}_{i_1}
\cdots z^{k_n}_{i_n}\otimes z^{j_1}_{k_1}\cdots z^{j_n}_{k_n}
\]
Hence the multiplication on $T(V^\ast\otimes V)_n^\ast$ is given by
\[
z^{k_1\ast}_{i_1}
\cdots z^{k_n\ast}_{i_n}\cdot z^{j_1\ast}_{l_1}\cdots z^{j_n\ast}_{l_n}=
\delta^{k_1}_{l_1}\cdots \delta^{k_n}_{l_n} z^{j_1\ast}_{i_i}\cdots z^{j_n\ast}_{i_n}
\]
By definition $z^j_i=x_j^\ast\otimes x_i$. Then $z^{j\ast}_i$ when
considered as an element of $V\otimes V^\ast$ is given by
$x_j\otimes x^\ast_i$.

It follows that the multiplication on $V^{\otimes n}\otimes (V^\ast)^{\otimes n}$ is given
by
\begin{equation}
\label{ref-4.4-31}
\bigl( x_{k_1}\cdots x_{k_n}\otimes x^\ast_{i_1}\cdots x^\ast_{i_n}\bigr)
\cdot
\bigl( x_{j_1}\cdots x_{j_n}\otimes x^\ast_{l_1}\cdots x^\ast_{il_n}\bigr)
=
\delta^{k_1}_{l_1}\cdots \delta^{k_n}_{l_n}
x_{j_1}\cdots x_{j_n}\otimes x^\ast_{i_1}\cdots x^\ast_{i_n}
\end{equation}
Since $x_{j_1}\cdots x_{j_n}\otimes x^\ast_{i_1}\cdots x^\ast_{i_n}$
corresponds to the endomorphism of $V^{\otimes n}$ which sends
$x_{i_1}\cdots x_{i_n}$ to $x_{j_1}\cdots x_{j_n}$ and all other basis
elements to zero, we see that \eqref{ref-4.4-31} corresponds precisely to
the opposite multiplication on $\End(V^{\otimes n})$.
\end{proof}
As $\fnd(A)$ is a graded algebra (see above) the multiplication induces vector space homomorphisms
\[
\fnd(A)_m\otimes \fnd(A)_n\r \fnd(A)_{m+n}
\]
which one checks to be coalgebra homomorphisms.

Hence there are dual algebra morphisms
\begin{equation}
\label{ref-4.5-32}
\mu_{m,n}:\fnd(A)^{\ast\circ}_{m+n}\r \fnd(A)^{\ast\circ}_m\otimes \fnd(A)^{\ast\circ}_n
\end{equation}
These are described in the next proposition.
\begin{proposition}
 Under the isomophism $\fnd(A)_n^{\circ\ast}\cong Z_n(A)$ given in
Proposition \ref{ref-4.2-29} the algebra morphism
\eqref{ref-4.5-32} fits in the following commutative diagram
\[
\xymatrix{
Z_{m+n}(A) \ar[r] \ar@{^(->}[d] & Z_m(A)\otimes Z_n(A)\ar@{^(->}[d]\\
 \End(V^{\otimes m+n})\ar[r]_-{\cong} & \End(V^{\otimes m})\otimes \End(V^{\otimes n})
}
\]
where the lower map is the obvious algebra isomorphisms
\end{proposition}
\begin{proof}
We have to verify the commutativity of the following diagram
\[
\xymatrix{
((V^\ast\otimes V)^{\otimes (m+n)})^\ast\ar[r]\ar[d] &
((V^\ast\otimes V)^{\otimes m})^\ast\otimes \ar[d]
((V^\ast\otimes V)^{\otimes n})^\ast \\
\End(V^{\otimes(m+n)})^\circ\ar[r] & \End(V^{\otimes m })^\circ \otimes \End(V^{\otimes n})^\circ
}
\]
This is again an easy verification, e.g.\ using explicit bases.
\end{proof}
By
Lemma \ref{ref-4.1-27},\ref{ref-4.2-28} we have
\[
\CoMod(\fnd(A))=\bigoplus_n \CoMod(\fnd(A)_n)=\bigoplus_n\Mod(\fnd(A)^{\ast\circ}_n)=
\bigoplus_n\Mod(Z_n(A))
\]
Furthermore  the induced monoidal structure on the righthand side is as follows:
let $W_1\in \Mod(Z_m(A))$, $W_2\in\Mod(Z_n(A))$. Then $W_1\otimes W_2\in \Mod(Z_{m+n}(A))$
is the pullback under $\mu_{m,n}$ of $W_1\otimes W_2$, considered as
$Z_m(A)\otimes Z_n(A)$ module

We deduce the well-know result (see e.g.\
\cite{HL1}).
\begin{proposition}
\label{ref-4.4-33}
Let $\Lscr_n(A)$ be the lattice of subspaces of $V^{\otimes n}$ generated
by $(R_i^{(n)})_i$. Then all the objects in $\Lscr_n(A)$ are $Z_n(A)$-modules and hence
$\fnd(A)$-comodules.
\end{proposition}
\begin{proof} The $R_i^{(n)}$ are obviously $Z_n(A)$ representations. Hence
so are all objects in $\Lscr_n(A)$.
\end{proof}
\section{Distributive algebras}
Let $A=TV/(R)$ be as above. Below we assume that $A$ is distributive.

For $I\subset \{1,\ldots,n-N+1\}$ we define
\[
C^{(n)}_I=R_I^{(n)}\biggl/\sum_{J\supsetneq I} R^{(n)}_{J}
\]
We will say that $I\subset \{1,\ldots,n-N+1\}$ is \emph{admissible} if $C_I\neq 0$.
Strictly speaking these definitions only make sense if $n\ge N$. For
$n<N$ we will put $\{1,\ldots,n-N+1\}=\emptyset$, and hence we only consider
$I=\emptyset$, $R_\emptyset=V^{\otimes n}$, $C_\emptyset=V^{\otimes n}$.

\begin{proposition}
\label{ref-5.1-34}
\begin{enumerate}
\item
The indecomposable projective $Z_n(A)$-representations are the $R^{(n)}_I$ for
$I$ admissible.
\item
The simple $Z_n(A)$-representations  are the $C^{(n)}_I$ for $I$ admissible.
\item $R^{(n)}_I$ is the projective cover of $C^{(n)}_I$.
\item For admissible $I,J$ we have
\[
\Hom_{Z_n(A)}(R^{(n)}_J,R^{(n)}_I)=
\begin{cases}
k &\text{if $J\supset I$}\\
0&\text{otherwise}
\end{cases}
\]
When $J\supset I$, the generator of
$\Hom_{Z_n(A)}(R^{(n)}_J,R^{(n)}_I)$ is given by the inclusion
$R_J^{(n)}\subset R^{(n)}_I$.
\end{enumerate}
\end{proposition}
\begin{proof}
Since the $(R^{(n)}_i)_i$ generate the distributive lattice $\Lscr_n(A)$   there exists a
 basis $(w_\alpha)_\alpha$ for $V^{\pt n}$  such
 that each vector space in $\Lscr_n(A)$ is spanned by a subset of the~$w_\alpha$ \cite[Prop.\ 7.1]{PP}).

Let $e_\alpha$ be the primitive idempotent in $\End(V^{\otimes n})$
corresponding to the basis vector $w_\alpha$. I.e.\ $e_\alpha$ is the
projection on $kw_\alpha$.  Then $e_\alpha$ preserves all $R_i^{(n)}$ and hence $e_\alpha\in Z_n(A)$.
Then the $(e_\alpha)_\alpha$ form still a
maximal set of orthogonal idempotents in~$Z_n(A)$.

For $I\subset \{1,\ldots,n-N+1\}$ let $\tilde{C}_I^{(n)}$ be the
subspace of $R^{(n)}_I$ spanned by the $w_\alpha$ which are not in
some $R_J$, $J\supsetneq I$. Clearly $C_I^{(n)}\cong
\tilde{C}^{(n)}_I$ and hence $\tilde{C}_I^{(n)}$ is non-zero if and
only if $I$ is admissible.

We have
\[
R_I^{(n)}=\bigoplus_{J\supset I} \tilde{C}^{(n)}_J
\]
It follows that
\[
\End(V^{\otimes n})=\bigoplus_{I,J} \Hom(\tilde{C}^{(n)}_I,\tilde{C}^{(n)}_J)
\]
and
\begin{equation}
\label{ref-5.1-35}
Z_n(A)=\bigoplus_{J\supset I} \Hom(\tilde{C}_I^{(n)},\tilde{C}_J^{(n)})
\end{equation}
Thus
\[
Z_n(A)/\rad(Z_n(A))=\bigoplus_{I\text{\ adm.}} \Hom(\tilde{C}_I^{(n)},\tilde{C}_I^{(n)})
\]
For every admissible $I$ pick an $\alpha_I$ such that $w_{\alpha_I}\in \tilde{C}^{(n)}_I$. Then
the indecomposable projectives in $Z_n(A)$-modules are given by
\[
Q_I=Z_n(A) e_{\alpha_I}
\]
We find
\[
Q_I=Z_n(A) e_{\alpha_I}=
\bigoplus_{J\supset I} \Hom(kw_{\alpha_I},\tilde{C}^{(n)}_J)=\Hom(kw_{\alpha_I},R^{(n)}_I)\cong R^{(n)}_I
\]
which proves (1).
We also find using the decomposition \eqref{ref-5.1-35}
\begin{align*}
\Hom_{Z_n(A)}(Q_J,Q_I)&=e_{\alpha_J}Z_n(A) e_{\alpha_I}\\
&=
\begin{cases}
k& \text{if $J\supset I$}\\
0& \text{otherwise}
\end{cases}
\end{align*}
which proves (4).

The simple top of $Q_I$ for $I$ admissible may be computed as the cokernel of
\[
\bigoplus_{Q_J\not\cong Q_I} Q_J \otimes \Hom_{Z_n(A)}(Q_J,Q_I)\r Q_I
\]
which is the same as the cokernel of
\[
\bigoplus_{R_J^{(n)}\subsetneq R_I^{(n)}} R^{(n)}_J \r R^{(n)}_I
\]
which is precisely $C_I^{(n)}$, proving (2) and (3).
\end{proof}
\begin{corollary}
\label{ref-5.2-36}
The category $\Mod(Z_n(A))$ is equivalent to the category
$\Rep(Q_{A,n-N+1})$ where $Q_{A,n-N+1}$ is the quiver with relations
which has a vertex $x_I$ for every admissible $I\subset
\{1,\ldots,n-N+1\}$ and arrows $x_{IJ}:x_I\r x_J$ for every $I\subset
J$.  The relations are given by $x_{JK}x_{IJ}=x_{IK}$. If $n-N+1\le 0$ then
$Q_{A,n-N+1}$ is a single vertex $x_\emptyset$ and no arrows.
\end{corollary}
\begin{proof} This is an easy consequence of the description of the indecomposable
projectives $Z_n(A)$-representations in Proposition \ref{ref-5.1-34}. The
functor realizing the stated equivalence is $\Hom_{Z_n(A)}
(\bigoplus_{I\text{ adm.}} R_I,-)$. If $P_I$ is the projective $Q_{A,n-N+1}$ representation
corresponding to the vertex $I$ then the inverse equivalence sends $P_I$ to~$R_I$.
\end{proof}
For use below we denote the inverse equivalence alluded to in the above proof by
$\Fscr^\circ_{A}$.

If we apply Theorem \ref{ref-3.1.1-3} with $\Cscr=\CoMod(\fnd(A)_n)$ then we
obtain an exact functor
\[
\Fscr_A:\Cube_{n-N+1}\r \CoMod(\fnd(A)_n)
\]
which sends $P_I^{(n-N+1)}$ to $R^{(n-N+1)}_I$
and which extends to
an exact monoidal functor
\[
\Fscr_A:\Cube_{\bullet}\r \CoMod(\fnd(A))
\]
The functor $\Fscr_A$ is an enhancement of the functor $F_A$ we used before as it fits in
the following commutative diagram
\[
\xymatrix{%
& \CoMod(\fnd(A)_n)\ar[dr]^{\text{forget}} &\\
\Cube_{n-N+1}\ar[ur]^{\Fscr_A}\ar[rr]_{F_A} &&\Vect
}
\]
\begin{proposition}
\label{ref-5.3-37}
The functor $\Fscr_A$ defines an equivalence
\[
\bar{\Fscr}_A:\Cube_{\bullet}/\Sscr_A\cong \CoMod(\fnd(A))
\]
where $\Sscr_A$ is
the localizing subcategory  of $\Cube_{\bullet}$ generated by  the
simples $S_I^{(l)}$ with $I$ not admissible.
\end{proposition}
\begin{proof}
For conciseness we drop some superscripts in the notations below, when they are clear from the context. We consider the following diagram
\begin{equation}\label{ref-5.2-38}
\xymatrix{ & \Rep(Q_{A,n-N+1})
\ar@<1ex>[dl]^{\Ind} \ar@{->}[dr]^{\Fscr^\circ_{A}}_{\cong}
& \\
\cu_{n-N+1} \ar[rr]_{\Fscr_A} \ar@<1ex>[ur]^\Res&   & \Mod(Z_n(A)) }
\end{equation}
Here $\Res$ is the restriction functor $\cu_n=\Rep(Q_n)\r \Rep(Q_{A.n})$
and $\Ind$ is the left adjoint to $\Res$.

We first claim $\Fscr^\circ_{A}=\Fscr_A\circ \Ind $.  To see this note that all
functors are right exact and commute with direct sums, so it is
sufficient to show that they take the same value on projectives. If
$I$ is admissible then we find $(\Fscr_A\circ \Ind)(P_I)=\Fscr_A(P_I)=R_I$ and we
find the same value for $\Fscr_{A}^\circ(P_I)$.

Now we claim $\Fscr_A=\Fscr_{A}^\circ\circ \Res$.
Let $M\in \Cube_{n-N+1}$. It follows that
\[
(\Fscr^\circ_{A}\circ\Res)(M)=\Fscr_A(\Ind\circ \Res(M))
\]
The canonical map
\[
\Ind\circ \Res(M)\r M
\]
is surjective and has its kernel in $\Sscr_A$. Hence since $\Fscr_A$ is exact
\[
\Fscr_A(\Ind\circ \Res(M))=\Fscr_A(M)
\]
which implies our claim.
Thus diagram \eqref{ref-5.2-38} is commutative in the two possible senses.

Put $\Sscr_{A,n}=\Sscr_A\cap \Cube_n$. Then $\Sscr_{A,n}$ lies in the kernel
of $\Res$ and hence in the kernel of $\Fscr_A$. Furthermore $\Res$
induces an equivalence $\Cube_n/\Sscr_{A,n}\cong \Rep(Q_{A,n})$.
Hence we obtain a commutative
diagram
\begin{equation}\label{ref-5.3-39}
  \xymatrix{ & \Rep(Q_{A,n-N+1})
    \ar@{-->}[dr]^{\Fscr^\circ_{A}}_{\cong}
    & \\
    \cu_{n-N+1}/\Sscr_{A,n-N+1} \ar[rr]_{\bar{\Fscr}_A} \ar@<1ex>[ur]^\Res_{\cong}&   &
\Mod(Z_n(A)) }
\end{equation}
Since $\Fscr^\circ_A$ is an equivalence, the same holds for $\bar{\Fscr}_A$.
\end{proof}
\begin{corollary} \label{ref-5.4-40} Assume that $A$ satisfies the
  extra condition (besides being distributive). Then the monoidal
  category $\CoMod(\fnd(A))$ is entirely determined by the numbers
  $(\dim J_n)_n$. If $N=2$ then $\CoMod(\fnd(A))$ is entirely
  determined by the Hilbert series of $A$.
\end{corollary}
\begin{proof}
  If follows from Proposition \ref{ref-5.3-37} that to describe the
  category  $\CoMod(\fnd(A))$  it suffices to
  know which $S^{(l)}_I$ are mapped to zero under $\Fscr_A$ or
  equivalently under~$F_A$.

To this end it is
sufficient to compute the dimension of $F_A S^{(l)}_I$. First we observe
that $S^{(l)}_I$ has a finite projective resolution
\[
\cdots \r \bigoplus_{\begin{smallmatrix} J\supset I\\
|J-I|=2\end{smallmatrix}}P^{(l)}_J \r\bigoplus_{\begin{smallmatrix}J\supset I
\\ |J-I|=1\end{smallmatrix}} P^{(l)}_J\r  S^{(l)}_I\r 0
\]
Hence to know $\dim F_A S^{(l)}_I$ it is sufficient to know $\dim F_A P^{(l)}_I$.

Since  $A$ satisfies the extra condition we have
$\dim  F_A P^{(l)}_I=\dim  F_A P^{(l)}_{\bar{I}}$ where
$\bar{I}$ is obtained from $I$ by filling up all holes of size $\le N-2$. Thus
we may assume that~$I$ has no holes of size $\le N-2$.

If on the other hand $I$ contains a hole of size $\ge N-1$ or else does not
contain $1$ or $l$ then we may write $P^{(l)}_I$ as a tensor product $P^{(l_1)}_{I_1}\otimes P^{(l_2)}_{I_2}$ for suitable $l_1,l_2,I_1,I_2$.  Hence we may reduce
to the case where $I=\{1,\ldots,l\}$. In that case $F_A P^{(l)}_I=J_l$. This
finishes the proof for general $N$.

If $N=2$ then it follows from the Koszul complex that the sets of numbers
$(\dim J_n)_n$ and $(\dim A_n)_n$ determine each other.
\end{proof}
\section{The symmetric algebra}
\label{ref-6-41}
In this section we assume that the ground field $k$ is algebraically closed
of characteristic zero. Below $A$ will be the symmetric algebra $SV=TV/(\wedge^2 V)$ of a vector space $V$ of
dimension $n$. As $SV$ is a quadratic algebra we will have $N=2$ in
this section.

To simplify things conceptually we will equip $V$ with an explicit
basis $(x_i)_i$ as in \S\ref{ref-4-24} and we put
$z^j_i=x^\ast_j\otimes x_i$. Let $\Oscr
=k[(z^j_i)_{i,j}]=S(\End(V)^\ast)$ be the coordinate ring of
$n\times n$-matrices. Then $\Oscr$ is a bialgebra which coacts on
$SV$ with the same formulas as \eqref{ref-4.1-25}\eqref{ref-4.2-26}.
Furthermore sending $z^j_i\mapsto z^j_i$ defines a bialgebra
homomorphism $B:\fnd(A)\r \Oscr$ compatible with the coactions on
$SV$. Hence we obtain a corresponding monoidal exact functor
\[
B_\ast:\CoMod(\fnd(A))\r \CoMod(\Oscr)
\]
which on the underlying vector spaces is the identity.

The bialgebra $\Oscr$ is semi-simple and its irreducible representations
are indexed by partitions of at most $n$ rows. For more details we refer to the excellent
text book \cite{FH}. If $\lambda$ is a partition of $m$ with conjugate partition $\lambda'$
then the associated irreducible $\Oscr$ (co)representation is
\[
S^\lambda V=\im\left(\bigotimes_i \wedge^{\lambda'_i} V\xrightarrow{\alpha}
V^{\otimes n}\xrightarrow{\beta} \bigotimes_j S^{\lambda_j} V\right)
\]
The precise form of the (anti-)symmetrization maps $\alpha$ and $\beta$ is derived from a
labeling by the numbers $1,\ldots,m$ of the Young diagram associated to $\lambda$
(which does not have to yield a standard tableau).
The representation $S^\lambda V$ is independent of this labeling.

For example if  $\lambda=[221]$  then a corresponding
labeled Young diagram is
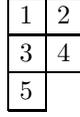
\setlength{\unitlength}{0.5cm}
\[
\begin{picture}(2,3)(0,0)
\put(0,0){\line(1,0){1}}
\put(0,1){\line(1,0){2}}
\put(0,2){\line(1,0){2}}
\put(0,3){\line(1,0){2}}

\put(0,0){\line(0,1){3}}
\put(1,0){\line(0,1){3}}
\put(2,1){\line(0,1){2}}

\put(0.3,0.3){\makebox{5}}
\put(0.3,1.3){\makebox{3}}
\put(0.3,2.3){\makebox{1}}
\put(1.3,1.3){\makebox{4}}
\put(1.3,2.3){\makebox{2}}
\end{picture}
\]
The corresponding irreducible $\Oscr$-representation $S^\lambda V$ is given by
\[
\im(\wedge^3 V\otimes \wedge^2 V\xrightarrow{\alpha} V^{\otimes 5}\xrightarrow{\beta} S^2V\otimes S^2V\otimes V)
\]
The image of $\alpha$ consists of  the tensors which are anti-symmetric in
the places $(1,3,5)$ and $(2,4)$. Likewise the map $\beta$ symmetrizes a tensor in
the places $(1,2)$, $(3,4)$ (and $5$ but this has no effect of course).

The construction of $S^\lambda V$ works also when $\lambda$ is replaced by a skew diagram $\lambda/\mu$, i.e.\ a
difference of two partitions $\lambda,\mu$.
In that case we have
\[
S^{\lambda/\mu} V=\im\left(\bigotimes_i \wedge^{\lambda'_i-\mu'_i} V\xrightarrow{\alpha}
V^{\otimes n}\xrightarrow{\beta} \bigotimes_j S^{\lambda_j-\mu_j} V\right)
\]
For example if $\lambda/\mu$ is the skew (labeled) diagram
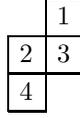
\[
\begin{picture}(2,3)(0,0)
\put(0,0){\line(1,0){1}}
\put(0,1){\line(1,0){2}}
\put(0,2){\line(1,0){2}}
\put(1,3){\line(1,0){1}}

\put(0,0){\line(0,1){2}}
\put(1,0){\line(0,1){3}}
\put(2,1){\line(0,1){2}}

\put(0.3,0.3){\makebox{4}}
\put(0.3,1.3){\makebox{2}}
\put(1.3,1.3){\makebox{3}}
\put(1.3,2.3){\makebox{1}}
\end{picture}
\]
then
\[
S^{\lambda/\mu}V=\im(\wedge^2 V\otimes \wedge^2 V\r V^{\otimes 4}\r V\otimes S^2 V\otimes V)
\]
In contrast to $S^\lambda V$ the representation $S^{\lambda/\mu}V$ is in general not irreducible but its decomposition
can be described in terms of the well-known Littlewood-Richardson coefficients.
\[
S^{\lambda/\mu}V=\sum_\nu (S^\nu )^{\oplus N_{\lambda\mu\nu}}
\]
The Littlewood-Richardson coefficients can be computed by means of a combinatorial recipe which
consists in counting certain admissible labelings of the boxes of~$\lambda/\mu$.

The following is the main observation of this section.
\begin{proposition} \label{ref-6.1-42} Assume
\[
I=\{1,\ldots,a,\widehat{a+1},\ldots,\widehat{a+b},a+b+1,\ldots,a+b+c,\ldots\}
\subset \{1,\ldots,m\}
\]
Then $B_\ast(\Fscr_A(S^{(m)}_I))=S^{\lambda/\mu}V$ where $\lambda/\mu$ is the rim hook
with $m+1$ boxes:
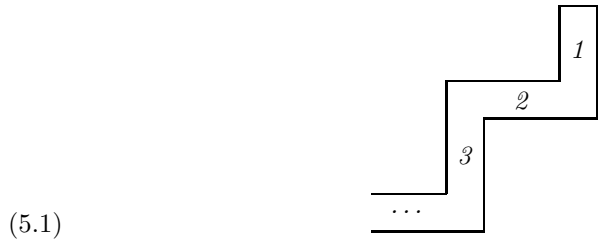
\begin{equation}
\label{ref-6.1-43}
\begin{picture}(6,6)(0,0)
\put(0,0){\line(1,0){3}}
\put(3,0){\line(0,1){3}}
\put(3,3){\line(1,0){3}}
\put(6,3){\line(0,1){3}}
\put(0,1){\line(1,0){2}}
\put(2,1){\line(0,1){3}}
\put(2,4){\line(1,0){3}}
\put(5,4){\line(0,1){2}}
\put(5,6){\line(1,0){1}}
\put(5.3,4.6){\makebox{1}}
\put(3.8,3.25){\makebox{2}}
\put(2.3,1.8){\makebox{3}}
\put(0.5,0.5){\makebox{\dots}}
\end{picture}
\end{equation}
such that the column labeled by ``$1$'' has $a+1$ boxes, the row labeled by ``$2$'' has
$b+1$-boxes, the row labeled by ``$3$'' has
$c+1$-boxes, etc\dots.

This result remains valid if $I$ has instead the form
\[
I=\{\widehat{1},\ldots,\widehat{b},b+1,\ldots,b+c,\ldots\}
\]
In that case we simply think of $a$ as being zero. This means that the
column labeled~$1$ contains only one box and thus is contained in
the row labeled $2$.
\end{proposition}
\begin{proof}
It is easy to see that
\begin{align*}
B_\ast(\Fscr_A(S^{(l)}_{\{1,\ldots,l\}}))&=B_\ast (\wedge^l V)=\wedge^l V\\
B_\ast(\Fscr_A(S^{(l)}_{\emptyset}))&=B_\ast (S^l V)=S^l V
\end{align*}
An arbitrary $S^{(m)}_I$ has a presentation
\[
\bigoplus_{i\not\in I} P^{(m)}_{I\cup \{i\}}\r P^{(m)}_I\r S^{(m)}_I\r 0
\]
Applying $\Fscr_A$ and using distributivity we get
\[
\Fscr_AS^{(m)}_I=R^{(m+1)}_{I}\left/\left(R^{(m+1)}_I \cap \sum_{i\not\in I} R_i^{(m+1)}\right) \right.
\]
Alternatively
\[
\Fscr_AS^{(m)}_I=\im \left(R^{(m+1)}_{I}\r V^{\otimes m+1}\r V^{\otimes m+1}\left/\sum_{i\not\in I} R_i^{(m+1)}\right.\right)
\]
We have
\[
R^{(m+1)}_{I}=\wedge^{a+1}V\otimes \wedge^{c+1}\otimes\cdots
\]
\[
V^{\otimes m+1}\left/\sum_{i\not\in I} R_i^{(m+1)}=S^{b+1}V\otimes S^{d+1}V\otimes\cdots\right.
\]
Thus $\Fscr_AS^{(m)}_I$ is $S^{\lambda/\mu}V$ where $\lambda/\mu$ is the diagram \eqref{ref-6.1-43}
numbered with the numbers $1,\ldots,m+1$ starting from the top right and ending at the bottom
left (thus this is not a standard tableau).
\end{proof}
\section{Some examples of distributive algebras}
In this section we describe some algebras which we could show to be distributive.
\subsection{Algebras derived from Koszul algebras}
We have the following easy general result.
\begin{propositions}
\label{ref-7.1.1-44}
Let $A=TV/(R)$ be a quadratic Koszul algebra and let $R'\in \Lscr_N(A)$. Then
 $A'=TV/(R')$ is a distributive algebra. If specifically we take
\begin{equation}
\label{ref-7.1-45}
 R'=\bigcap_{i=1}^{N-1} V^{\otimes i-1}\otimes R\otimes V^{N-i-1}
\end{equation}
 then $A'$ also satisfies the
extra condition and hence is $N$-Koszul.
\end{propositions}
\begin{proof} The distributivity of  $A'$ follows immediately from the distributivity
of $A$. To check the extra condition we have to verify for $1\le l\le p$ with $p\le N-1$
\begin{multline}
\label{ref-7.2-46}
\left(\bigcap_{i=1}^{N-1} V^{\otimes i-1}\otimes R\otimes V^{N+(p-1)-i-1}\right)
\cap
\left(\bigcap_{i=1}^{N-1} V^{\otimes (p-1)+i-1}\otimes R\otimes V^{N-i-1}\right)
\\\subset
\bigcap_{i=1}^{N-1} V^{\otimes (l-1)+i-1}\otimes R\otimes V^{N-(l-1)-i-1}
\end{multline}
The lefthand side of this equation is equal to
\[
\bigcap_{j=1}^{p+N-2} V^{\otimes j-1}\otimes R\otimes V^{N+p-2-j}
\]
We claim that each of the terms on the righthand side of
\eqref{ref-7.2-46} appears in this intersection. To prove this we
have to find for any $i=1,\ldots,N-1$ a $j\in \{1,\ldots,p+N-2\}$
such that $j-1=l+i-2$. Thus we have to take $j=l+i-1$. One verifies
that indeed $1\le l+i-1\le p+N-2$.
\end{proof}
\begin{examples} Take $R=\wedge^2 V$ and $R'$ as \eqref{ref-7.1-45}. Then $R'=\wedge^N V$
and hence $A'=TV/(\wedge^N V)$ is the $N$-generalization of the
symmetric algebra introduced in~\cite{Berger}. By Proposition
\ref{ref-7.1.1-44} we see that $A'$ is distributive and satisfies
the extra condition (and hence is $N$-Koszul by \eqref{ref-3.4-11}).
These facts were proved by Berger in \cite{Berger} using
``confluence''. See below.
\end{examples}
\subsection{Confluence}
\label{ref-7.2-47}
For $N>2$ the only tool to establish distributivity of $TV/(R)$ with $R\subset V^{\otimes N}$
seems to be ``confluence'' \cite{Berger,Bergman1}. For the benefit of the reader we give
a quick introduction to this concept.

Let $W$ be a finite dimensional vector space equipped with a fixed
totally ordered basis $X$. If $S\subset W$ is a subvector space then
an element $x$ of $X$ is called non-reduced with respect to $S$ if
$S$ contains an element of the form $x-\sum_{y<x}c_y y$. We denote
the set of such non-reduced monomials by $\NRed(S)$. The following
is easy to see.
\begin{lemmas} \begin{enumerate}
\item $\dim S=|\NRed(S)|$.
\item If $S\subset R$ then $\NRed(S)\subset \NRed(R)$.
\end{enumerate}
\end{lemmas}
Thus $\NRed(-)$ is an order preserving map from the lattice of subvector spaces
of $W$ to the lattice of subsets of $X$.

From the fact that $\NRed(-)$ is order preserving we immediately deduce
\begin{align}
\NRed(R\cap S)&\subset \NRed(R)\cap \NRed(S)\label{ref-7.3-48}\\
\NRed(R+ S)&\supset \NRed(R)\cup \NRed(S)\label{ref-7.4-49}
\end{align}
The following result was proved by Berger \cite{Berger2}.
\begin{lemmas}
If one of these inclusions is an equality then so is the other.
\end{lemmas}
\begin{definitions} We say that $R,S\subset W$ are confluent if one (and hence both) of the inclusions
\eqref{ref-7.3-48}\eqref{ref-7.4-49} is an equality.
\end{definitions}
The following is one of the main fact about confluence \cite{Berger2}.
\begin{theorems}
\label{ref-7.2.4-50}
If $\Rscr=\{(R_i)_i\}$ is a collection of pairwise confluent subspaces of $W$ then
$\NRed(-)$ defines a isomorphism between the lattice $\Lscr$ of subspaces generated by $R_i$
and the lattice of subsets of $X$ generated by $\NRed(R_i)$. In particular $\Lscr$ is distributive.
\end{theorems}
A typical application of this concept is the following. Let $A=TV/(R)$ with $R\subset V^{\otimes N}$. Let $X$ be a totally ordered basis for $V$. We equip $V^{\otimes n}$
with the basis $X^{\times n}$, ordered lexicographically. Then we say that $A$
is confluent (with respect to $X$) if for all $n$ and all $i$ the subspaces
$R^{(n)}_i$ of $V^{\otimes n}$ are pairwise confluent. Using Theorem \ref{ref-7.2.4-50}
we get
\[
\text{confluent}\Rightarrow \text{distributive}
\]
The following was observed
by Berger \cite{Berger2}.
\begin{propositions}\label{ref-7.2.4-50bis} In order for $A$ to be confluent with respect to $X$ it is necessary
and sufficient that $R\otimes  V^{\otimes i}$ and $V^{\otimes
i}\otimes R$ are confluent inside $V^{\otimes (i+N)}$ for
$i=1,\ldots,N-1$.
\end{propositions}
\subsection{The (deformed) Yang-Mills algebras}
In this section we assume that $k$ has characteristic zero.  The
Yang-Mills algebras and their deformations were introduced by Connes
and Dubois-Violette in \cite{CDV1,CDV2}. After a short reminder of how
they are constructed we will show that they are distributive.

The $i$'th cyclic derivative  on $F=k\langle x_1,\ldots,x_n\rangle$
is a $k$-linear map
\[
\frac{{}^\circ\partial}{\partial x_i}:F/[F,F]\r F
\]
which on monomials is defined by
\[
\frac{{}^\circ\partial m}{\partial x_i}=\sum_{m=ux_iv}vu
\]
Assume that $V$ is an $n$-dimensional vector space equipped with a basis $(x_i)_i$. To
avoid trivialities we assume $n\ge 2$.
If $w\in TV/[TV,TV]$ is homogeneous then the Jacobian algebra associated to $w$
is the algebra $A=TV/(R)$ where $R$ is the vectorspace spanned by the cyclic derivatives
of $w$.  It is easy to see that $R$ and hence $A$ does not depend on the choice of the basis $(x_i)_i$.

One has
\[
\bar{w}\overset{\text{def}}{=}\sum_i x_i\frac{{}^\circ\partial w}{\partial x_i}=\sum_i
\frac{{}^\circ\partial w}{\partial x_i}x_i\in R\otimes V\cap V\otimes
R
\]
The map $w\mapsto \bar{w}$ gives an isomorphism between $(F/[F,F])_n$ and
the cyclically invariant elements of $F_n=V^{\otimes n}$.

\medskip

Assume now that $V$ is equipped with a non-degenerate symmetric
bilinear form $(-,-)$ and put $g_{ij}=(x_i,x_j)$. The inverse of the
matrix $(g_{ij})_{ij}$ is denoted by $(g^{ij})_{ij}$. Let $O(V)$ be
the corresponding orthogonal group. Then one has the following
result
\begin{lemmas} The space $(TV/[TV,TV])_4^{O(V)}$ is two-dimensional and is spanned by the elements
\[
w_1=\sum_{i,j,p,q} g^{ip}g^{jq}[x_i,x_j][x_p,x_q]
\]
\[
w_2=\biggl(\sum_{i,j} g^{ij} x_ix_j\biggr)^2
\]
\end{lemmas}
\begin{proof} To prove this we may assume that $k=\bar{k}$ and that
  $(x_i)_i$ is an orthogonal basis for $V$. In other words
  $g_{ij}=\delta_{ij}$.  We first consider $(V^{\otimes
    4})^{O(V)}$. The first fundamental theorem of invariant theory for
  the orthogonal groups implies that $(V^{\otimes 4})^{O(V)}$ is
  spanned by
\begin{align*}
s_1&=\sum_{i,j}x_ix_ix_jx_j\\
s_2&=\sum_{i,j}x_ix_jx_ix_j\\
s_3&=\sum_{i,j}x_ix_jx_jx_i
\end{align*}
Since $n\ge 2$ it is clear that $s_1,s_2,s_3$ are linearly independent. Hence they
form a basis for $(V^{\otimes 4})^{O(V)}$.

Since $O(V)$ is reductive the quotient map $(V^{\otimes 4})^{O(V)}\r (TV/[TV,TV])_4^{O(V)}$ is surjective
and hence $(TV/[TV,TV])_4^{O(V)}$ is spanned by $s_1,s_2,s_3$. It is clear that
$s_1$ and $s_3$ are equal modulo commutators but $s_1$ and $s_2$ remain independent.  We
now use the fact that $w_2=s_1$ and $w_1=2(s_2-s_3)$.
\end{proof}
By definition the deformed Yang-Mills algebra $A_\lambda$ ($\lambda\in
k$) is the Jacobian algebra for the potential $w_1+\lambda w_2$. The equations
of $A_\lambda$ have the form
\[
\sum_{j,p} g^{jp}\bigl([x_j,[x_i,x_p]]+\lambda \{ x_i,x_jx_p\}\bigr)
\]
where $\{a,b\}=ab+ba$. The ordinary Yang-Mills algebra is $A_0$.

\begin{theorems} The deformed Yang-Mills algebra $A_\lambda$ is distributive
and satisfies the extra condition.
\end{theorems}
\begin{proof}
We may clearly assume that the ground field is algebraically closed. Then we
may assume that $g^{ij}$ has the following form
\[
g^{ij}=
\begin{cases}
1&\text{if $i+j=n$}\\
0&\text{otherwise}
\end{cases}
\]
Put $\bar{\imath}=n-i$.  The equations of $A_\lambda$ now have
the form
\[
\sum_{j} [x_j,[x_i,x_{\bar{\jmath}}]]+\lambda \{ x_i,x_jx_{\bar{\jmath}}\}
\]
Thus we get
\begin{align*}
\NRed(R)=\{x_nx_{i}x_1 \mid  1\le i \le n \}
\end{align*}

We claim that $A_\lambda$ is confluent. By Proposition
\ref{ref-7.2.4-50bis} 
we must check the equalities
\[
\NRed(R_1^{(n)}\cap R^{(n)}_{n-2})=\NRed(R_1^{(n)})\cap \NRed(R^{(n)}_{n-2})
\]
for $n=4,5$. We already know that the sought equalities are inclusions
of the type ``$\subset$''. Hence it is sufficent to prove
\begin{equation}
\label{ref-7.5-51}
\dim R_1^{(n)}\cap R^{(n)}_{n-2}\ge |\NRed(R_1^{(n)})\cap \NRed(R^{(n)}_{n-2})|
\end{equation}
Since $\NRed(R_1^{(5)})\cap \NRed(R^{(5)}_{3})=\emptyset$ there is nothing
to prove for $n=5$.

On the other hand $\NRed(R_1^{(4)})\cap
\NRed(R^{(4)}_{2})=\{x_nx_nx_1x_1\}$ and $\bar{w}_\lambda\in
R_1^{(4)}\cap R^{(4)}_{2}$. Thus $\dim R_1^{(4)}\cap R^{(4)}_{2}\ge
1$ and so \eqref{ref-7.5-51} holds for $n=4$ as well.

To check the extra condition, using Theorem \ref{ref-7.2.4-50} me must verify
\[
\NRed(R^{(5)}_1)\cap \NRed(R^{(5)}_3)\subset \NRed(R^{(5)}_2)
\]
As the lefthand side is empty there is nothing to prove.
\end{proof}
Using \eqref{ref-3.4-11} we recover the following result proved in \cite{CDV1,CDV2}.
\begin{corollarys} The deformed Yang-Mills algebra $A_\lambda$ is $3$-Koszul.
\end{corollarys}
\subsection{Three dimensional cubic Artin-Schelter  regular algebras}
Recall \cite{AS} that an AS-regular algebra is a graded $k$-algebra
$A=k{+}A_1{+}A_2{+}\cdots$ satisfying the following conditions
\begin{enumerate}
\item $\dim A_{i}$ is bounded by a polynomial.
\item The projective dimension  of $k$ is finite.
\item There is exactly one $i$ for which  $\Ext^i_A(k,A)$ is non-vanishing
and for this $i$ we have $\dim \Ext^i_A(k,A)=1$.
\end{enumerate}
Three dimensional regular algebras generated in degree one were classified
in \cite{AS,ATV1,ATV2} and in general in \cite{Steph1,Steph2}. It was
discovered that they are intimately connected to plane elliptic curves.

There are two possibilities for a three dimensional regular
algebra $A$ generated in degree one.
\begin{enumerate}
\item $A$ is defined by three generators satisfying three quadratic relations
(the ``quadratic case'').
\item $A$ is defined by two generators satisfying two cubic relations
(``the cubic case'').
\end{enumerate}
In \cite{AS} it is shown that all 3-dimensional regular algebras
are obtained by specialization from a number ``generic'' regular algebras.
These generic regular algebras depend on at most two parameters.

Quadratic three dimensional AS-regular algebras are Koszul and hence
distributive.  Cubic three dimensional AS-regular algebras are
$3$-Koszul so one would expect it should be easy to verify distributivity
for them. Nonetheless we have not found a clean way to establish this.

Below we discuss so-called Type A algebras by using a generalization
of the notion of $I$-type introduced in \cite{TVdB}. It is likely that
the other types can be handled in a similar way but we have not
carried out the required verifications.

Type A algebras are associated to a
triple $(E,\sigma,\Lscr)$ where $E$ is a smooth elliptic curve, $\sigma$ is a translation
and $\Lscr$ is a line bundle of degree two on $E$. Put $V=H^0(E,\Lscr)$ and
\[
R=\ker(H^0(E,\Lscr)^{\otimes 3}\xrightarrow{a\otimes b\otimes c\mapsto
  ab^\sigma c^{\sigma^2}} H^0(E,\Lscr\otimes \Lscr^\sigma\otimes
  \Lscr^{\sigma^2}))
\]
(where $(-)^\sigma=\sigma^\ast$).  Then the cubic three-dimensional
AS-regular algebra associated to $(E,\sigma,\Lscr)$ is given by
$TV/(R)$.

We prove
\begin{propositions} \label{ref-7.4.1-52}
Let $A$ be a cubic Type A regular algebra. Then $A$ is distributive.
\end{propositions}
It follows easily from \cite[Thm 6.11]{AS} that a cubic Type A regular algebra
is not confluent in the sense \S\ref{ref-7.2-47} for any choice of basis $X$. Our arguments
still use confluence but in a more sophisticated way. 

We need some preparatory work. There is a well-defined map
\[
\phi:A_n\r H^0(E,\Lscr\otimes \cdots\otimes \Lscr^{\sigma^{n-1}}):
\overline{a_0\otimes \cdots \otimes a_n}\mapsto a_0a_1^{\sigma}\cdots a_{n-1}^{\sigma^{n-1}}
\]
If $f\in A_n$ then
we write $(f)$ for the divisor of the image of $\phi(f)$ in
$H^0(E,\Lscr\otimes\cdots\otimes \Lscr^{\sigma^{n-1}})$.
If $f\in A_m$, $g\in A_n$ then
\[
(fg)=(\phi(fg))=(\phi(f)\phi(g)^{\sigma^m})=(f)+\sigma^{-m}(g)
\]
Let $x$, $y$ be a basis of $V=H^0(E,\Lscr)$. Put $(x)=P+P'$,
$(y)=Q+Q'$. We write $x_n$, $y_n$ for arbitrary non-zero elements in $V$ with
divisors $\sigma^n P+\sigma^{-n}P'$, $\sigma^n Q+\sigma^{-n} Q'$.

We will say that $P,P',Q,Q'$ are generic if they have pairwise disjoint orbits.
\begin{lemmas} \label{ref-7.4.2-53} Assume that $P,P',Q,Q'$ are generic and $\sigma^4\neq \Id$.
Then there are non-zero scalars $\alpha$, $\beta$, $\gamma$, $\delta$ such
that $R$ has a basis given by
\begin{equation}
\label{ref-7.6-54}
\begin{aligned}
y_0x_1x_{-2}&-\alpha x_0x_{-3} y_{2}-\beta x_0y_1x_{-2}\\
y_0y_{-3} x_{2}&-\gamma x_0y_1y_{-2} -\delta y_0x_1 y_{-2}
\end{aligned}
\end{equation}
\end{lemmas}
\begin{proof} We need to prove that the elements \eqref{ref-7.6-54} are
  linearly independent and contained in $R$. We first show that they
  are linearly independent. Assume this is not the case. I.e.\ there
  are scalars $\lambda$, $\mu$, not both zero such that
\[
\lambda(y_0x_1x_{-2}-\alpha x_0x_{-3} y_{2}-\beta x_0y_1x_{-2})
+\mu (y_0y_{-3} x_{2}-\gamma x_0y_1y_{-2} -\delta y_0x_1 y_{-2})=0
\]
in $V^{\otimes 3}$.
Collecting the terms starting with $x_0$
\begin{align*}
-\lambda\alpha x_{-3}y_2-\lambda\beta y_1x_{-2}-\mu \gamma y_1y_{-2}&=0
\end{align*}
By our genericity assumption $x_{-3}$ and $y_1$ are linearly independent. Thus we get
\begin{align*}
-\lambda\alpha y_2&=0\\
-\lambda\beta x_2-\mu \gamma y_{-2}&=0
\end{align*}
From the first equation we deduce $\lambda=0$ and from the second $\mu=0$.

Now we show that the elements \eqref{ref-7.6-54} are indeed contained
in $R$ (for suitable choices of scalars). To this end we have to show
that the element
\[
y_0x_1^\sigma x_{-2}^{\sigma^2}
\] of $H^0(E,\Lscr\otimes\Lscr^\sigma\otimes \Lscr^{\sigma^2})$) is a
linear combination with non-zero coefficients of
\[
x_0x^\sigma_{-3} y^{\sigma^2}_{2}, \quad x_0y_1^\sigma x_{-2}^{\sigma^2}
\]
and similarly the element
\[
y_0y^\sigma_{-3} x^{\sigma^2}_{2}
\]
is a linear combination  with non-zero coefficients of
\[
x_0y_1^\sigma y_{-2}^{\sigma^2},
\quad y_0x^\sigma_1 y^{\sigma^2}_{-2}
\]
We compute the divisors of all these elements
\begin{align*}
y_0x_1^\sigma x_{-2}^{\sigma^2}:&\qquad Q+Q'+P+\sigma^{-2}P'+\sigma^{-4}P+P'\\
x_0x^\sigma_{-3} y^{\sigma^2}_{2}:&\qquad P+P'+\sigma^{-4}P+\sigma^2P'+Q+\sigma^{-4} Q'\\
x_0y_1^\sigma x_{-2}^{\sigma^2}:&\qquad P+P'+Q+\sigma^{-2}Q'+\sigma^{-4}P+P'\\
y_0y^\sigma_{-3} x^{\sigma^2}_{2}:&\qquad Q+Q'+\sigma^{-4}Q+\sigma^2Q'+P+\sigma^{-4} P'\\
x_0y_1^\sigma y_{-2}^{\sigma^2}:&\qquad P+P'+Q+\sigma^{-2}Q'+\sigma^{-4}Q+Q'\\
 y_0x^\sigma_1 y^{\sigma^2}_{-2}:&\qquad Q+Q'+P+\sigma^{-2}P'+\sigma^{-4}Q+Q'
\end{align*}
All these divisors are different so these elements are pairwise linearly independent.
We also find
\[
y_0x_1^\sigma x_{-2}^{\sigma^2}, \,\, x_0x^\sigma_{-3} y^{\sigma^2}_{2},\,\, x_0y_1^\sigma x_{-2}^{\sigma^2}
\in H^0(E,\Lscr\otimes \Lscr^\sigma\otimes \Lscr^{\sigma^2}(-P-P'-Q-\sigma^{-4}P))
\]
By Riemann Roch the righthand side is a two dimensional vector
space. Hence $y_0x_1^\sigma x_{-2}^{\sigma^2}$, $x_0x^\sigma_{-3}
y^{\sigma^2}_{2}$, $x_0y_1^\sigma x_{-2}^{\sigma^2}$ are linearly dependent. Since they
are not scalar multiples of each other the first one is indeed a linear combination
of the other two and the coefficients are non-zero.

Similarly we find
\[
 y_0y^\sigma_{-3} x^{\sigma^2}_{2}, \,\, x_0y_1^\sigma y_{-2}^{\sigma^2},
\,\, y_0x^\sigma_1 y^{\sigma^2}_{-2}\in H^0(E,\Lscr\otimes \Lscr^\sigma\otimes \Lscr^{\sigma^2}
(-P-Q-Q'-\sigma^{-4}Q))
\]
With the same reasoning as above we deduce that $y_0y^\sigma_{-3}
x^{\sigma^2}_{2}$ is a linear combination of $x_0y_1^\sigma
y_{-2}^{\sigma^2}$, $y_0x^\sigma_1 y^{\sigma^2}_{-2}$ with non-zero
coefficients.
\end{proof}
\begin{proof}[Proof of Proposition \ref{ref-7.4.1-52}]
We will give the proof in the case that $\sigma^4\neq \Id$. If $\sigma^4=\Id$
then $A$ is a so-called ``linear'' algebra which can be easily analyzed directly
(see \cite[Prop.\ 7.4]{ATV1}).

We may clearly assume that our base field is so large that we can select generic
$P,P',Q,Q'$ as in the discussion above.

Let us grade the monomials in $x,y$ by putting $\deg x=(1,0)$, $\deg
y=(0,1)$. For monomials $\mu$ in $x,y$ and $(a,b)\in \ZZ^2$ we
define $v_{a,b}(\mu)\in TV$ by constructing a path in the grid below
starting at $(a,b)$ and ending at $(a,b)+|\mu|$ where the variables
$x$, $y$ indicate which branch we should take at each vertex.
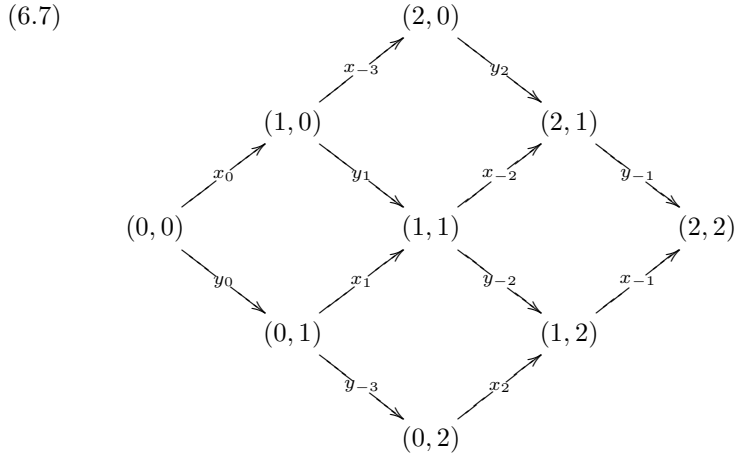
\begin{equation}
\label{ref-7.7-55}
\xymatrix{
&&(2,0)\ar[dr]|{y_2}\\
&(1,0)\ar[ru]|{x_{-3}}\ar[dr]|{y_1} 
&&(2,1)\ar[dr]|{y_{-1}}\\
(0,0)\ar[ru]|{x_0}
\ar[dr]|{y_0}&&
(1,1)
\ar[ru]|{x_{-2}}\ar[dr]|{y_{-2}}&&(2,2)\\
&(0,1) \ar[dr]|{y_{-3}}\ar[ru]|{x_1}
&&(1,2)\ar[ru]|{x_{-1}} &&&\\
&&(0,2)\ar[ru]|{x_2} &&&&
}
\end{equation}
One square looks like
\[
\xymatrix{
& (a+1,b) \ar[dr]^{y_{-3b+a+1}}&\\
(a,b)\ar[ru]^{x_{-3a+b}}\ar[rd]_{y_{-3b+a}} && (a+1,b+1)\\
&(a,b+1) \ar[ru]_{x_{-3a+b+1}}&
}
\]
Thus for example
\[
v_{0,0}(xyyx)=x_0y_1y_{-2}x_{-1}
\]
It is clear that $v_{a,b}(\mu)_\mu$ is a basis for $TV$ for any $a,b$. We extend
$v_{a,b}$ linearly to an isomorphism of graded vector spaces $TV\r TV$.
If
$(a,b)=(0,0)$ then we write $v(\mu)=v_{a,b}(\mu)$.

We claim that $V^{\otimes a}\otimes R\otimes V^{\otimes b}$ has a basis of the
form
\begin{gather}
v(\mu yxx \mu')-\alpha_{\mu} v(\mu xxy\mu')-\beta_{\mu} v(\mu xyx\mu')\\
v(\mu yyx \mu')-\gamma_{\mu} v(\mu xyy\mu')-\delta_{\mu} v(\mu yxy\mu')
\end{gather}
where $\mu$, $\mu'$ run through the monomials of length $a$, $b$ and $\alpha_\mu$,
$\beta_\mu$, $\gamma_\mu$, $\delta_\mu$ are scalars depending on $\mu$.

By Lemma \ref{ref-7.4.2-53} we find that $R$ has a basis of the type
\begin{align*}
r'_{m,n}&=y_n x_{m+1}x_{m-2}-\alpha_{m,n} x_m x_{m-3}y_{n+2}-\beta_{m,n} x_m y_{n+1} x_{m-2}\\
s'_{m,n}&=y_n y_{n-3}x_{m+2}    -\gamma_{m,n}
    x_m y_{n+1}y_{n-2}
-\delta_{m,n} y_n x_{m+1} y_{n-2}
\end{align*}
for each $m,n$. Put
\begin{align*}
r_{m,n}&=y xx-\alpha_{m,n} x xy-\beta_{m,n} x y x\\
s_{m,n}&=yyx-\gamma_{m,n} xy y-\delta_{m,n} yx y
\end{align*}
If $(m,n)=(-3p+q,-3q+p)$ then
\begin{align*}
v_{p,q}(r_{m,n})&=r'_{m,n}\\
v_{p,q}(s_{m,n})&=s'_{m,n}
\end{align*}

Assume now that $\mu,\mu'$ are monomials of degree $(p,q)$, $(p',q')$ respectively
 with $p+q=a$, $p'+q'=b$.
Put $(m,n)=(-3p+q,-3q+p)$. Then
\begin{align*}
v(\mu r_{m,n}\mu')&=v(\mu) v_{p,q}(r_{m,n})v_{p+2,q+1}(\mu')\\
v(\mu s_{m,n}\mu')&=v(\mu) v_{p,q}(s_{m,n})v_{p+1,q+2}(\mu')
\end{align*}
Thus $v(\mu r_{m,n}\mu')$, $v(\mu s_{m,n}\mu')\in V^{\otimes a}\otimes
R\otimes V^{\otimes b}$. It is clear that the elements $\mu
r_{m,n}\mu'$, $\mu s_{m,n}\mu'$ for varying $\mu,\mu'$ are all
linearly independent and their number is equal to the dimension of
$V^{\otimes a}\otimes R \otimes V^{\otimes b}$. So they yield a basis for
$V^{\otimes a}\otimes R\otimes V^{\otimes b}$ after applying
$v(-)$. This finishes the proof of our claim.

We now claim that the $(R_a^{(n)})_a$ are pairwise confluent inside
$V^{\otimes n}$ when the latter is equipped with the ordered basis
$v(\mu)_\mu$ where~$\mu$ runs through the lexicographically ordered
set of monomials of length $n$.  The proposition then follows from
Theorem \ref{ref-7.2.4-50}.

We compute for $1\le a<b\le n-3$
\[
|\NRed(R_a^{(n)})\cap \NRed(R_b^{(n)})|
=
\begin{cases}
2^{n-4}&\text{if $b=a+1$}\\
0&\text{if $b=a+2$}\\
2^{n-4}&\text{if $b\ge a+3$}
\end{cases}
\]
If $b\ge a+3$ then
\[
\dim R_a^{(n)}\cap R_b^{(n)}=2^{n-4}
\]
Hence the only statement that needs to be shown is
\[
\dim R_a^{(n)}\cap R_{a+1}^{(n)}=2^{n-4}
\]
This follows from the fact that for a cubic three-dimensional
AS-regular algebra we have $\dim (R\otimes V\cap V\otimes R)=1$
\cite[Prop. (2.4)]{AS}.

\def\cprime{$'$} \def\cprime{$'$} \def\cprime{$'$}
\ifx\undefined\bysame
\newcommand{\bysame}{\leavevmode\hbox to3em{\hrulefill}\,}
\fi

\end{proof}
\end{document}